\documentclass[12pt]{amsart}

\usepackage{amsmath}
\usepackage{url}
\usepackage{color}

\usepackage{amsfonts, amsthm, amsmath, amssymb}

\theoremstyle{plain}
\newtheorem{thm}{Theorem}[section]
\newtheorem{lem}[thm]{Lemma}
\newtheorem{prop}[thm]{Proposition}
\newtheorem{cor}[thm]{Corollary}

\newtheorem{defn}[thm]{Definition}

\newtheorem{ex}[thm]{Example}

\def\ds{\displaystyle}
\def\cal{\mathcal}

\def\dd {{\rm d}}
\def\dmax{{\rm d}_{\rm max}}
\def\adj{{\rm adj}}
\def\ord{{\rm ord}}

\DeclareMathOperator{\N}{\mathbb N}
\DeclareMathOperator{\dx}{d_{\max}}
\DeclareMathOperator{\Ap}{Ap}
\DeclareMathOperator{\maxap}{maxAp} 
\newcommand{\pre}{\preceq}
\newcommand{\st}{:}

\def\red#1 {\textcolor{red}{#1 }}
\def\blue#1 {\textcolor{blue}{#1 }}
\begin{document}

\title{The Maximal Denumerant of a Numerical Semigroup}
\author{Lance Bryant and James Hamblin}
\date{}
\maketitle

\begin{abstract}
\noindent Given a numerical semigroup $S = \langle a_0, a_1, a_2, \ldots, a_{t} \rangle$ and $n\in S$, we consider the factorization $n = c_0 a_0 + c_1 a_1 + \cdots + c_{t} a_{t}$ where $c_i\ge0$.  Such a factorization is {\em maximal} if $\sum c_i$ is a maximum over all such factorizations of $n$.  We provide an algorithm for computing the maximum number of maximal factorizations possible for an element in $S$, which is called the {\em maximal denumerant} of $S$.  We also consider various cases that have connections to the Cohen-Macualay and Gorenstein properties of associated graded rings for which this algorithm simplifies. 
\end{abstract}

\section{Introduction}

Let $\N$ denote the nonnegative integers. A {\em numerical semigroup} $S$ is a subsemigroup of $\N$ that contains 0 and has a finite complement in $\N$. For two elements $u$ and $u'$ in $S$, write $u \pre u'$ if there exists an $s\in S$ such that $u+s = u'$. This defines a partial ordering on $S$. The minimal elements in $S\setminus\{0\}$ with respect to this ordering form the unique {\em minimal set of generators of $S$}, which is denoted by $\{a_0,a_1,a_2,\dots, a_{t}\}$ where $a_0<a_1<a_2<\dots<a_{t}$.  

The numerical semigroup $S=\{\sum_{i=0}^{t} c_ia_i : c_i\ge 0\}$ is represented using the notation $S = \langle a_0, a_1, a_2, \ldots, a_{t} \rangle$. Since the minimal generators of $S$ are distinct modulo $a_0$, the set of minimal generators is finite. Furthermore, $S$ having finite complement in $\N$ is equivalent to $\gcd\left(a_0,a_1,\ldots ,a_{t}\right)= 1$.  The cardinality, $t+1$, of the set of minimal generators of a semigroup $S$ is called the {\em embedding dimension of $S$}, written $\nu=\nu(S)$. The element $a_0$ is called the {\em multiplicity} of $S$, also sometimes written $e=e(S)$. When $S\ne \N$, we have $2\le\nu\le e$.  

By definition, if $n\in S=\langle a_0,a_1,\dots,a_t\rangle$, then there exists a $(t+1)$-tuple of nonnegative integers ${\bf c}=(c_0,c_1,\ldots,c_{t})$ such that $\sum c_i a_i = n$.  We call $\bf c$ a {\em factorization} of $n$ in $S$, or simply an $S$-factorization of $n$.  The {\em length} of $\bf c$, written $|{\bf c}|$, is $\sum c_i$.  We say that ${\bf c}$ is {\em maximal} if no $S$-factorization of $n$ has length greater than $|{\bf c}|$, and {\em minimal} if no $S$-factorization of $n$ has length less then $|{\bf c}|$. For $n \in S$, the {\em order} of $n$, written $\ord(n;S)$ or $\ord(n)$, is the length of a maximal $S$-factorization of $n$. Similarly, the length of a minimal $S$-factorization of $n$ is denoted by $\min\ord(n;S)$.

Provided $S\ne \N$, it is a finitely generated monoid which fails to have unique factorization with respect to its minimal generators. For a given element of $s\in S$, a basic arithmetic constant that measures this failure is the cardinality of the set of factorizations of $s$. This is called the {\em denumerant} of $s$ in $S$, denoted by $d(s;S)$. See \cite{R} for an exhaustive view of related results. This is one of several numerical invariants that have appeared in recent papers exploring the factorization properties of numerical semigroups; for example, see \cite{ACHP, AG1, AG2, BCKR, CDHK, CHK, CGLM}.

In \cite{BHJ}, the authors considered a variation of the denumerant. 

\begin{defn}
The {\em maximal denumerant} of an element $s$ in $S$ is the number of factorizations of $s$ that have maximal length and is denoted by $\dd_{\max}(s; S)$. The {\em maximal denumerant of the semigroup $S$} is defined to be $\dmax(S)=\max_{s\in S}\{\dd_{\max}(s; S)\}.$
\end{defn} 

It was shown in \cite{BHJ} that the maximal denumerant of a semigroup $S$ is always finite, and the following formulas were given for this value when the embedding dimension of $S$ is less than four: Let $S$ be (perhaps non-minimally) generated by $a_1$, $a_2$, and $a_3$, and let $g=\gcd(a_2-a_1,a_3-a_1)$, $m=(a_2-a_1)/g$, and $n=(a_3-a_1)/g$. Then we can write
$$S=\langle a_1,a_1+ gm,a_1+ gn\rangle.$$
If $0\le \alpha < mn$ such that $\alpha \equiv -a_1  \,\,\mathrm{mod} \,\,mn$, we have
\begin{eqnarray}\label{eqn1}
\dx(S) &=& \begin{cases} \ds
\left\lceil\frac{a_1}{mn}\right\rceil, & \ds  \text{if } \alpha \in  \langle m,n\rangle \\ \\ \ds
\left\lceil\frac{a_1}{mn}\right\rceil +1, & \text{otherwise.}
\end{cases}
\end{eqnarray}

\noindent Moreover, if $x$ and $y$ are integers such that $mx+ny=a_1$, we have
\begin{eqnarray}
\ds \dd_{\max}(S) &=& \left\lceil \frac{x}{n}\right\rceil + \left\lceil \frac{y}{m} \right\rceil.
\end{eqnarray}

It is not clear how these formulas might extend to semigroups with higher embedding dimension, and in this paper we are concerned with computing the maximal denumerant of an arbitrarily given semigroup. Notice that the formulas above make use of the values $a_1$, $a_2-a_1$, and $a_3-a_1$. These integers generate the {\em blowup semigroup} of $S$, see \cite{BDF}. This suggests that the blowup semigroup may be useful for computing the maximal denumerant in general. In Section 3, we show that this is indeed the case. In essence, the problem of counting maximal factorizations in a semigroup $S$ corresponds to counting factorizations with bounded length in the blowup semigroup, and Theorem \ref{th:lowerbound} is the crucial result.

Section 4 is concerned with a special class of semigroups, namely additive semigroups (also called good \cite{Ba} and $M$-additive \cite{Bry2}). These are the semigroups for which the associated graded ring with respect to the maximal ideal of the the corresponding local semigroup ring is Cohen-Macaulay, see \cite{Ba,BF, BF2,Bry,Bry2,DMS,Gar}. Supersymmetric semigroups (also called $M$-symmetric \cite{Bry}) are important examples of additive semigroups since the associated graded ring is Gorenstein in this case, see \cite{Bry,Bry2,DMS,DMS2}. The method for computing the maximal denumerant outlined in Section 3 simplifies when the semigroup is additive, and especially when the semigroup is supersymmetric. 

In the next section, we begin by establishing the basic connection between factorizations in a semigroup and its blowup that will be used throughout the paper.

\section{Factorizations and the blowup semigroup}

Let $S=\langle e, a_1,a_2,\dots, a_t\rangle$ be a numerical semigroup with mutiplicity $e$ and embedding dimension $\nu = t+1$, where $e<a_1< a_2 < \cdots<a_t$. As noted in the introduction, the $a_i$ are distinct modulo $e$. They are, in fact, the least elements of $S$ in their respective congruence classes. The set consisting of the least elements in $S$ for each congruence class modulo $e$ is called {\em the Ap\'ery set} of $S$. More generally, an {\em Ap\'ery set} with respect to an element $u\in S$ is defined to be $\Ap(S; u) = \{w\in S \st w-u\not\in S\}$. Additionally, we denote the elements of $S$ congruent to $i$ modulo $e$ by $S_i$, so that $\Ap(S; e) = \{\min(S_i) \st 0\le i < e\}$.

For our study of the maximal denumerant of $S$, we consider the blowup of $S$, sometimes called the {\em Lipman} semigroup of $S$ in honor of \cite{L}.  

\begin{defn}
Given a numerical semigroup $S = \langle e, a_1, a_2, \ldots, a_t \rangle$, the {\em blowup} of $S$ is the semigroup $B = \langle e, a_1-e, a_2-e, \ldots, a_t-e \rangle$.  \end{defn}

We will write $d_i = a_i - e$, so that $B = \langle e, d_1, d_2, \ldots, d_t \rangle$,  and set ${\cal D} = \{ e, d_1, d_2, \ldots, d_t \}$.

Note that the element $e$ is not necessarily the multiplicity of $B$, and $\nu=t+1$ is not necessarily the embedding dimenstion of $B$. For example, if we have $S=\langle 4,5,6 \rangle$ with multiplicity 4 and embedding dimesnion 3, then $B=\langle 4, 1, 2\rangle =  \N$ with multiplicity 1 and embedding dimension 1. In general, we have $e(B)\le e(S)$ and $\nu(B)\le \nu(S)$. Thus $\cal D$ is always a generating set of $B$, but not necessarily minimal. 

We wish to consider $B$ with respect to the generating set $\mathcal D$, so we say that $(x_0,x_1,\ldots,x_t)$ is a $B^{\cal D}$-factorization of $b\in B$ if $x_0e+\sum_{i=1}^t x_i d_i = b$.
For example, $B=\mathbb N$ is the blowup of  $S=\langle 4,5,6 \rangle$ and $\mathcal D = \{4,1,2\}$. Thus, $(0,3,1)$ is a $B^{\mathcal D}$-factorization of $5$ since $5=0(4) + 3(1) + 1(2)$, and $d(5;B^{\mathcal D}) = 4$ is the denumerant of $5$ in $B$ with respect to $\cal D$ since there are a total of four $B^{\mathcal D}$-factorizations of $5$.

\begin{lem} \label{th:biglemma}
With the definitions above, let $s \in S$ and $r \in \N$.
\begin{enumerate}
\item If ${\bf x} = (x_0,x_1,\ldots,x_{t})$ is an $S$-factorization of $s$ with length $r$, then $(0,x_1,x_2,\ldots,x_{t})$ is a $B^{\cal D}$-factorization of $s-re$ with length at most $r$.
\item If ${\bf y} = (y_0,y_1,\ldots,y_{t})$ is a $B^{\cal D}$-factorization of $s-re$ with length at most $r$, then $(2y_0+r-|{\bf y}|,y_1,y_2,\ldots,y_{t})$ is an $S$-factorization of $s$ with length $r+y_0$.  In particular, $\ord(s;S) \geq r$.
\end{enumerate}
\end{lem}

\begin{proof} For the proof of {\it 1.}, let ${\bf x}=(x_0, x_1, \ldots, x_{t})$ be an $S$-factorization of $s$ with length $r$.  Then 
\begin{eqnarray*}
s & = & x_0 e + x_1 a_1 + \cdots + x_{t} a_{t} \\
s-re & = & (x_0 e - x_0 e) + (x_1 a_1 - x_1 e) + \cdots + (x_{t}a_{t} - x_{t}e) \\
s-re & = & 0 e + x_1 d_1 + \cdots + x_{t} d_{t}
\end{eqnarray*}
Hence $(0,x_1,x_2,\ldots,x_{t})$ is a $B^{\cal D}$-factorization of $s-re$ with length at most $r$.

For the proof of {\it 2.}, let ${\bf y}=(y_0,y_1,\ldots,y_{t})$ be a $B^{\cal D}$-factorization of $s-re$ with $|{\bf y}|\leq r$.  Hence
\begin{equation} \label{eq:dpart}
s - re = y_0 e + y_1 d_1 + \cdots + y_{t} d_{t} .
\end{equation}
Let $v = |{\bf y}|-y_0$ and note that $v\leq r$.  Now add $ve$ to both sides of (\ref{eq:dpart}) to obtain
\begin{equation} \label{eq:dpart2}
s - re + ve = y_0 e + y_1 a_1 + \cdots + y_{t} a_{t}.
\end{equation}
Next, adding $(r-v)e$ to both sides of (\ref{eq:dpart2}) gives
\[
s = (y_0 + r - v)e + y_1 a_1 + y_2 a_2 + \cdots + y_{t} a_{t}.
\]
Since $y_0+r-v = 2y_0+r-|{\bf y}|$, this tells us that $(2y_0+r-|{\bf y}|,y_1,y_2,\ldots,y_{t})$ is an $S$-factorization of $s$ with length $y_0 + r - v + \sum_{i=1}^{t} y_i = y_0 + r$. 
\end{proof}

\section{The Maximal Denumerant}

Recall that if $s\in S$ has a maximal $S$-factorization ${\bf x}$, then $|{\bf x}| = \ord(s)$.  By Lemma \ref{th:biglemma} we have a corresponding $B^{\cal D}$-factorization of $s - \ord(s)e$.  This situation will occur frequently in this section, so we have the following definition.

\begin{defn}
Given $s\in S$, the {\em adjustment} of $s$ is $\adj(s)=s-\ord(s)e$. Morevoer, for $U\subset S$, we set $\adj(U) = \{ \adj(s) : s\in U \}$ and call this the {\em adjustment} of $U$. 
\end{defn}

Using Lemma \ref{th:biglemma}, we have a way to find $\dmax(s;S)$ via $B^{\mathcal D}$-factorizations. This is the content of Proposition \ref{th:dmaxchar}.

\begin{lem} \label{th:1stzero}
Let $s\in S$.  If $(x_0,x_1,\ldots,x_{t})$ is a $B^{\cal D}$-factorization of $\adj(s)$ with length at most $\ord(s)$, then $x_0 = 0$.
\end{lem}

\begin{proof}
By statement {\it 2.}\ of Lemma \ref{th:biglemma}, we see that $s$ has an $S$-factorization of length $x_0 + \ord(s)$.  By definition, this implies $x_0 = 0$.
\end{proof}

\begin{prop} \label{th:dmaxchar}
If $s \in S$, then $\dmax(s;S)$ is the number of $B^{\cal D}$-factorizations of $\adj(s)$ of length at most $\ord(s)$.
\end{prop}

\begin{proof} Given a $t$-tuple $(x_1,\ldots,x_t)$, write $x_0 = \ord(s) - \sum_{i=1}^t x_i$.
It is sufficient to prove that $(x_0,x_1,\ldots,x_{t})$ is a maximal $S$-factorization of $s$ if and only if $(0,x_1,\ldots,x_{t})$ is a $B^{\cal D}$-factorization of $\adj(s)$ with length at most $\ord(s)$.

Let $(x_0,x_1,\ldots,x_{t})$ be a maximal $S$-factorization of $s$.  Then by statement {\it 1.}\ of Lemma \ref{th:biglemma} with $r = \ord(s)$, we see that $(0,x_1,\ldots,x_{t})$ is a $B^{\cal D}$-factorization of $\adj(s)$ with length at most $\ord(s)$.

Conversely, let $(0,x_1,\ldots,x_{t})$ be a $B^{\cal D}$-factorization of $\adj(s)$ with length at most $\ord(s)$.  Note that by Lemma \ref{th:1stzero}, any such factorization has a zero in the first component.  By statement {\it 2.}\ of Lemma \ref{th:biglemma}, since $\sum_{i=0}^{t} x_i = \ord(s)$, we see that $(x_0,x_1,\ldots,x_{t})$ is a maximal $S$-factorization of $s$.
\end{proof}

From Proposition \ref{th:dmaxchar} we know that  $\dmax(S)$ can be computed by considering the $B^{\cal D}$-factorizations of elements in $\adj(S)$. The rest of this section details how to carry out this computation.

\begin{prop} \label{th:nvsnplusa}
Let $f\in \Ap(B;e)$, $s= f + \min\ord(f;B^{\mathcal D})e$, and $s' = f + \ord(f;B^{\mathcal D})e$. Then $s$ and $s'$ are in $S$, and 
\begin{enumerate}
\item $\adj(s+ke)= f$, for $k\ge0$.
\item $\dmax(s'+ke) = d(f;B^{\mathcal D})$, for $k\ge0$.
\end{enumerate}
\end{prop}

\begin{proof}
For {\it 1.}, by Lemma \ref{th:biglemma}, we have that $\ord(s+ke;S) \geq \min\ord(f;B^{\cal D})$.  Since $(s+ke)-(\min\ord(f;B^{\cal D})+k+1)e = f-e\notin B$, it follows by Lemma \ref{th:biglemma} that $s+ke$ does not have an $S$-factorization of length greater than $\min\ord(f;B^{\cal D})$, hence $\ord(s+ke;S) = \min\ord(f;B^{\cal D})$.  Hence $\adj(s+ke)=(s+ke)-(\min\ord(f;B^{\cal D})+k)e = f$.

For {\it 2.}, notice that $\adj(s') = f$ by {\it 1.} Using Proposition \ref{th:dmaxchar}, it suffices to show that $\ord(s'+ke;S) \ge \ord(f;B^{\mathcal D})$. From the proof of {\it 1.}, we know that $\ord(s'+ke;S) = \min\ord(f;B^{\mathcal D}) + k + (s'-s)/e = \min\ord(f;B^{\mathcal D}) + k +\ord(f;B^{\mathcal D}) -\min\ord(f;B^{\mathcal D}) = \ord(f;B^{\mathcal D}) + k $.
\end{proof}

\begin{cor} 
The adjustment of $S$ is finite, and $\Ap(B;e) \subset \adj(S) \subset B$.
\end{cor}

\begin{proof} 
That $\Ap(B;e) \subset \adj(S)$ follows from statement {\it 1.} of Proposition \ref{th:nvsnplusa}. By \cite{BDF}, $B=\{s-ke \st \ord(s;S) \ge k\ge 1\}$, and so $\adj(S) \subset B$.
\end{proof}

Now we define two sets of $B^{\cal D}$-factorizations that we need to consider.
Recall that $S_i$ is the set of elements of $S$ that are congruent to $i$ modulo $e$.

\begin{defn}
Fix $0\leq i < e$ and write $\adj(S_i) = \{ u_0, u_1, u_2, \ldots \}$ in increasing order.  For any $b\in B$, write ${\cal P}(b)$ for the set of all $B^{\cal D}$-factorizations of $b$.  
For $u_j \in S_i$, define ${\cal R}(u_j)$ recursively:
\begin{itemize}
\item ${\cal R}(u_0) = {\cal P}(u_0)$
\item If $j>0$, then $${\cal R}(u_j) = \left\{ {\bf x}\in {\cal P}(u_j) : 
|{\bf x}| < \min\ord(u_{j-1};B^{\cal D})-\frac{u_j-u_{j-1}}{e} \right\}.$$
\end{itemize}
\end{defn}

With these definitions, if $u\in \adj(S)$, then ${\cal P}(u)$ is the set of all $B^{\cal D}$ factorizations of $u$, and ${\cal R}(u)$ is a subset of ${\cal P}(u)$ that contains only those $B^{\cal D}$-factorizations that have a certain bounded length.  As we will see, it is exactly these sets ${\cal R}(u)$ that allow us to compute $\dmax(S)$.

\begin{thm} \label{th:upperbound}
If $n\in S$, then $\dmax(n;S) \leq |{\cal R}(\adj(n))|$.  In particular, each set ${\cal R}(\adj(n))$ is nonempty.
\end{thm}

\begin{proof} Let $n\in S$ and write $n\bmod a = i$.
By Proposition \ref{th:dmaxchar}, we have that $\dmax(n;S)$ is the number of $B^{\cal D}$-factorizations of $\adj(n)$ of length at most $\ord(n)$.  So it suffices to show that if $\bf x$ is a $B^{\cal D}$-factorization of $\adj(n)$ of length at most $\ord(n)$, then ${\bf x}\in {\cal R}(\adj(n))$.

As before, write $\adj(S_i)=\{ u_0, u_1, u_2, \ldots \}$ in increasing order, so that $\adj(n) = u_j$ for some $j$.  Let ${\bf x}\in {\cal P}(u_j)$ with $|{\bf x}|\leq \ord(n)$.

If $j=0$, then ${\bf x}\in {\cal P}(u_j) = {\cal R}(u_j)$ as desired.

If $j>0$, write $m=\min\ord(u_{j-1};B^{\cal D})$.  Since $u_{j-1}<u_j$, we have $u_{j-1} = n-ra$ for some $r > \ord(n)$.  
It follows from Lemma \ref{th:biglemma} that $u_{j-1}$ has no $B^{\cal D}$-factorizations of length at most $r$.  Hence $m>r$.  Also,
\[
u_j - u_{j-1} = (n-\ord(n)e) - (n-re) = (r-\ord(n))e.
\]
Thus
\[
m-\frac{u_j-u_{j-1}}{e} = m-(r-\ord(n)) > \ord(n).
\]
Therefore $|{\bf x}| \leq \ord(n) < m - \frac{u_j-u_{j-1}}{e}$ and so ${\bf x}\in {\cal R}(u_j)$ as desired. 
\end{proof}

The set ${\cal R}(u_j)$ is defined recursively in terms of ${\cal R}(u_{j-1})$.  In the following lemma we see that we can consider the elements of ${\cal R}(u_j)$ in terms of ${\cal R}(u_k)$ for any $k<j$.

\begin{lem} \label{th:helper}
Let $0\leq i < e$ and write $\adj(S_i) = \{ u_0, u_1, u_2, \ldots \}$ in increasing order.  Let $j>0$ and $0\leq k < j$, and let ${\bf x}\in {\cal R}(u_j)$.  Then $|{\bf x}| < \min\ord(u_k;B^{\cal D})-\frac{u_j-u_k}{e}$.
\end{lem}

\begin{proof} By the definition of ${\cal R}(u_j)$, we have $|{\bf x}| < \min\ord(u_{j-1};B^{\cal D})-\frac{u_j - u_{j-1}}{e}$.  If $k = j-1$, then we are done.  If not, then note that $u_{j-1} = u_k + re$ where $r = \frac{u_{j-1} - u_k}{e} > 0$.  

Write $m=\min\ord(u_k;B^{\cal D})$.  Let ${\bf y}=(y_0,y_1,\ldots,y_{t}) \in {\cal P}(u_k)$ with $|{\bf y}| = m$.  Now $(y_0 + r, y_1,\ldots,y_{t})$ is a $B^{\cal D}$-factorization of $u_{j-1}$ with length $m+r$, and so $m+r \geq \min\ord(u_{j-1};B^{\cal D})$.  Now
\begin{eqnarray*}
|{\bf x}| & < & \min\ord(u_{j-1};B^{\cal D})-\frac{u_j - u_{j-1}}{e} \\
& \leq & m+r-\frac{u_j - u_{j-1}}{e} \\
&=& m - \frac{u_j-u_k}{e},
\end{eqnarray*}
which is the desired result.
\end{proof}

We are now ready to prove the connection between maximal denumerants and the $\cal R$ sets.

\begin{thm} \label{th:lowerbound}
For all $u \in \adj(S_i)$, there exists $s\in S_i$ with $\dmax(s;S) = |{\cal R}(u)|$.  
\end{thm}

\begin{proof} Let $u_j \in S_i$.  Let $r$ be the length of the longest factorization in ${\cal R}(u_j)$, and let $s = u_j + re$.  Since $s-re = u_j$ has a $B^{\cal D}$-factorization of length at most $r$, by Lemma \ref{th:biglemma} we have $\ord(s) \geq r$.

Suppose $\ord(s) > r$.  Then $\adj(s) = u_k$ for some $k<j$.  From Lemma  \ref{th:helper} we have $r < \min\ord(u_k;B^{\cal D}) - \frac{u_j-u_k}{e}$.  Now $u_j = s-re$ and $u_k = s-\ord(s)e$, so we have
$r < \min\ord(u_k;B^{\cal D}) - (\ord(s)-r)$, and so $\min\ord(u_k;B^{\cal D}) > \ord(s)$.  So the shortest $B^{\cal D}$-factorization of $\adj(s)$ has length greater than $\ord(s)$.  This is a contradiction, since by Proposition \ref{th:dmaxchar} we have that $\adj(s)$ always has at least one $B^{\cal D}$-factorization of length at most $\ord(s)$.  This contradiction proves that $\ord(s)=r$, and so $u_j = \adj(s)$.

Now by Proposition \ref{th:dmaxchar}, we know that $\dmax(s;S)$ equals the number of $B^{\cal D}$-factorizations of $u_j$ of length at most $r$.  However, since $r$ is the length of the longest factorization in ${\cal R}(u_j)$, it follows that ${\cal R}(u_j)$ is the set of all $B^{\cal D}$-factorizations of $u_j$ of length at most $r$.  Thus $\dmax(s;S) = |{\cal R}(u_j)|$. 
\end{proof}

The next corollaries follow immediately from Theorems \ref{th:upperbound} and \ref{th:lowerbound}.

\begin{cor} \label{th:dmaxsi}
Let $S$ be a semigroup with multiplicity $e$, and let $0\leq i < e$.  Then 
$\dmax(S_i) = \max_{u\in \adj(S_i)} |{\cal R}(u)| $.
\end{cor}

\begin{cor}
Let $S$ be a semigroup with the notation defined above.
\[
\dmax(S) = \max_{0\leq i<e,u\in \adj(S_i)} |{\cal R}(u)| .
\]
\end{cor}

We will now show an example to illustrate how to compute $\dmax(S)$ working with $B^{\cal D}$-factorizations.

\begin{ex}
Let $S = \langle 15, 17, 36, 38, 71 \rangle$.
\end{ex} 

For this example, we will let $i=11$ and compute $\dmax(S_{11})$.  The computation of $\dmax(S_i)$ for the other values of $i$ is similar.

Now ${\cal D} = \{ 15, 2, 21, 23, 56 \}$ and $B = \langle 15, 2, 21, 23, 56 \rangle$.  
The element of $\Ap(B;15)$ congruent to 11 mod 15 is $f_{11}=26$.  To compute $\adj(S_{11})$, we start with the smallest element of $S_{11}$ (which is 71) and work up to $f_{11}+\min\ord(f_{11};B^{\cal D})$, by Lemma \ref{th:nvsnplusa}.  Since $f_{11}$ only has one $B^{\cal D}$-factorization (of length 13), we have the following computation:

\begin{center}
\begin{tabular}{l|c|c|c|c|c|c|c|c|c|c|c|}
$s\in S_{11}$ & 71 & 86 & 101 & 116 & 131 & 146 & 161 & 176 & 191 & 206 & 221 \\ \hline
$\ord(s)$ & 1 & 2 & 3 & 4 & 5 & 6 & 7 & 8 & 10 & 11 & 13 \\ \hline
$\adj(s)$ & 56 & 56 & 56 & 56 & 56 & 56 & 56 & 56 & 41 & 41 & 26 \\ \hline
\end{tabular}
\end{center}

We see that $\adj(S_{11}) = \{ 26, 41, 56 \}$.  

Now ${\cal R}(26) = {\cal P}(26)$ is the set of all $B^{\cal D}$-factorizations of 26, of which there is only one: $(0,13,0,0,0)$.  So $M(26) = 13$.

Next, ${\cal R}(41)$ is the set of $B^{\cal D}$-factorizations of 41 with length less than $\min\ord(26;B^{\cal D})-\frac{41-26}{15} = 12$.  There are two of these: $(0,9,0,1,0)$ and $(0,10,1,0,0)$.  So $\min\ord(41;B^{\cal D})=10$.

Finally, ${\cal R}(56)$ is the set of $B^{\cal D}$-factorizations of 56 with length less than $10-\frac{56-41}{15} = 9$.  There are three of these: $(0,0,0,0,1)$, $(0,5,0,2,0)$, and $(0,6,1,1,0)$.

The largest of these three sets is $|{\cal R}(56)|=3$, so by Corollary \ref{th:dmaxsi} we have $\dmax(S_{11})=3$.

\section{Additive Semigroups}

When computing the maximal denumerant of $S$ using the set $\adj(S) \subset B$, we are concerned with $B^\mathcal{D}$ factorizations with bounded length. However, for elements of $\adj(S)$ contained in $\Ap(B;e)$, this restriction is removed and we consider the denumerant of these elements (with respect to $\mathcal D$). Thus, we have this result: If $\adj(S) = \Ap(B;e)$, then $$\dmax(S) = \max\{d(f;B^{\mathcal{D}})\st f\in \Ap(B;e)\}.$$ We can refine this even more. Recall from the introduction that $u \pre u'$ if there exists an $s\in S$ such that $u+s = u'$. This defines a partial ordering on $S$.

\begin{defn}
For a semigroup $S$ and element $u\in S$, we define $\maxap(S;u) = \{w\in \Ap(S;u) \st w$ is maximal in $\Ap(S;u)\setminus\{0\}$ with respect to $\pre\}$.
\end{defn}

\begin{prop}\label{dmax and maxap}
Let $S$ be a semigroup with multiplicity $e$ and blowup $B$. If $\adj(S) = \Ap(B;e)$, then $$\dmax(S) = \max\{d(f;B^{\mathcal{D}})\st f\in \maxap(B;e)\}.$$
\end{prop}

\begin{proof}
Let $g\in \Ap(B;e)\setminus \maxap(B;e)$. It suffices to show that there exists an element $f \in \maxap(B;e)$ such that $d(f;B^{\mathcal{D}}) \ge d(g;B^{\mathcal{D}})$. To do this, we choose $f \in \maxap(B;e)$ so that $f = g + g'$ where $g' \in \Ap(B;e)$. Let ${\bf C = \{c_i\}}$ be the set of $B^{\mathcal D}$-factorizations of $g$ and fix a factorization ${\bf d}$ of $g'$. The ${\bf c_i} + {\bf d}$ is a $B^{\mathcal D}$-factorization of $f$ for each $i$. Furthermore, ${\bf c_i} = {\bf c_j}$ if and only if $i=j$. Thus, we have $d(f;B^{\mathcal{D}}) \ge d(g;B^{\mathcal{D}})$.
\end{proof}

Proposition \ref{dmax and maxap} can greatly reduce the amount of work needed to compute the maximal denumerant of a semigroup $S$, especially in light of the result contained in Corollary \ref{blowup sym}, which states when we need only compute the denumerant of a single element of $B$. This follows from Proposition \ref{dmax and sym} found in \cite{Bry}.

\begin{defn}
The {\em Frobenius number}, $F(S)$, of a semigroup $S$ is the largest integer not contained in $S$. The semigroup $S$ is called {\em symmetric} if whenever $x+y = F(S)$ for $x,y\in \mathbb Z$, then exactly one of $x$ and $y$ belongs to $S$.
\end{defn}

It is easy to verify that $\mathbb N$ is symmetric, and it is classically known that any semigroup with embedding dimension 2 is symmetric, see \cite{Br,S}. For such semigroups, there is only one maximal element in an Ap\'ery set with respect to the partial ordering $\pre$.

\begin{prop}\label{dmax and sym}
Let $S$ be a semigroup, $u\in S$, and $\Ap(S;u) = \{w_0,w_1,\dots,w_{u-1}\}$ where $w_0<w_1<\dots<w_{u-1}$. Then the following are equivalent.
\begin{enumerate}
\item $S$ is symmetric
\item $w_i+w_j=w_{u-1}$ whenever $i+j=u-1$
\item $w \pre w_{u-1}$ for all $w\in \Ap(S;u)$
\item $\maxap(S;u) = \{F(S) + u\}$
\end{enumerate}
\end{prop}

\begin{cor}\label{blowup sym}
Let $S$ be a semigroup with multiplicity $e$ and blowup $B$. If $\adj(S) = \Ap(B;e)$ and $B$ is symmetric, then $\dmax(S) = d(F(B)+e;B^{\mathcal{D}})$.\end{cor}

The next definition distinguishes an important class of numerical semigroups, which turn out to be precisely the semigroups we are considering in this section.

\begin{defn}
A numerical semigroup $S$ with multiplicity $e$ is {\em additive} if
\begin{equation} \label{eq:adddef}
\ord(u+e;S) = \ord(u;S)+1
\end{equation}
for all $u\in S$.
\end{defn}

Additive semigroups are semigroups for which the associated graded ring $gr_m(R) = \bigoplus_{i=0}^\infty m^i/m^{i+1}$ of the corresponding ring $(R,m) = k[[x^e,x^{a_1},x^{a_2},\dots,x^{a_t}]]$ with respect to the maximal ideal $m$ is Cohen-Macaulay. 

\begin{prop}\label{adj and add}
A semigroup $S$ with multiplicity $e$ and blowup $B$ is additive if and only if $\adj(S) = \Ap(B;e)$.
\end{prop}

\begin{proof}
First assume that $S$ is additive and let $f_i \in \Ap(B;e)$ where $f_i \equiv i \mod e$. By Proposition \ref{th:dmaxchar}, $s_i = f_i + \min\ord(f_i;B^{\mathcal D})e$ is an element of $S$ with $\adj(s_i) = f_i$. For any integer $k$, such that $s_i + ke \in S$, we have $\adj(s_i+ke) = (s_i + ke) - \ord(s_i+ke; S)e = (s_i + ke) - \ord(s_i; S)e -ke = s_i - \ord(s_i; S)e = \adj(s_i) = f_i$. It follows that $\adj(S_i) = f_i$. Since this holds for all $0\le i<e$, we have $\adj(S) = \Ap(B;e)$.

Now assume that $\adj(S) = \Ap(B;e)$ and let $u\in S$. Choose $f \in \Ap(B;e)$ such that $u \equiv f\mod e$. Then $u - \ord(u;S)e = f = (u+e) -\ord(u+e;S)e$, and we have $\ord(u+e;S) = \ord(u;S) + 1$.
\end{proof}

We can now restate Proposition \ref{dmax and maxap} and Corollary \ref{blowup sym} as follows.

\begin{thm}\label{add thm}
Let $S$ be an additive semigroup with multiplicity $e$ and blowup $B$. Then $\dmax(S) = \max\{d(f;B^{\mathcal{D}})\st f\in \maxap(B;e)\}$.
\end{thm}

\begin{cor}\label{sym cor}
Let $S$ be an additive semigroup with multiplicity $e$ and symmetric blowup $B$. Then $\dmax(S) = d(F(B)+e;B^{\mathcal{D}})$.
\end{cor}

Before looking at some examples, there is an important class of additive semigroups in connection to ring theory for which Corollary \ref{sym cor} applies. 

\begin{defn}
Let $S$ be a semigroup with multiplicity $e$ and $\Ap(S;e) = \{w_0,w_1,\dots,w_{e-1}\}$ where $w_0<w_1<\dots <w_{e-1}$. The semigroup $S$ is called {\em supersymmetric} if $S$ is additive and, in addition, $w_i +w_j = w_{e-1}$ and $\ord(w_i;S) +\ord(w_j;S) = \ord(w_{e-1};S)$ whenever $i+j = e-1$.
\end{defn}

Supersymmetric semigroups are semigroups for which the associated graded ring $gr_m(R) =\bigoplus_{i=0}^\infty m^i/m^{i+1}$ of the corresponding ring $(R,m) = k[[x^e,x^{a_1},x^{a_2},\dots,x^{a_t}]]$ with respect to the maximal ideal $m$ is Gorenstein.

By definition, a supersymmetric semigroup $S$ is additive, and it follows from Proposition \ref{dmax and sym} that $S$ is symmetric. We will show that the blowup semigroup is also symmetric. 

\begin{lem}
Let $S$ be a supersymmetric semigroup. Then the blowup semigroup of $S$ is symmetric.
\end{lem}

\begin{proof}
Since $S$ is additive, by Proposition \ref{adj and add}, $\adj(\Ap(S;e)) \subset \adj(S) = \Ap(B;e)$. These sets have the same cardinality, so, in fact, $\adj(\Ap(S;e)) = \Ap(B;e)$. 

Consider $f\in \Ap(B;e)$ where $e$ is the multiplicity of $S$. From above, we know that $f = w_i -\ord(w_i)e$ for some $w_i \in \Ap(S;e)$. Since $S$ is supersymmetric, we have $[w_i -\ord(w_i)e] + [w_j -\ord(w_j)e] = w_{e-1} -\ord(w_{e-1})e$ where $i+j=e-1$. Furthermore $w_j -\ord(w_j)e = \adj(w_j)$ and $w_{e-1} -\ord(w_{e-1})e=\adj(w_{e-1})$ are both in $\Ap(B;e)$. Thus $f \pre adj(w_{e-1})$ for all $f\in \Ap(B;e)$. This forces $\adj(w_{e-1})$ to be the largest element of $\Ap(B;e)$, and by Proposition \ref{dmax and sym}, $B$ is symmetric.
\end{proof}

\begin{thm}
Let $S$ be a supersymmetric semigroup with multiplicity $e$ and blowup $B$. Then $\dmax(S) =  d(F(B)+e;B^{\mathcal{D}})$.
\end{thm}

We conclude with two more results.

\begin{prop}
If $S$ is additive and $B=\mathbb N$, or equivalently, the difference between the two smallest minimal generators is 1, then $\dmax(S)$ is the denumerant of $e-1$ in $B$ with respect to $\mathcal D$. 
\end{prop}

\begin{proof}
This follows from Corollary \ref{sym cor} since $B$ is symmetric and $F(B) = -1$.
\end{proof}

\begin{prop}\label{ex3}
If $S$ is generated by an arithmetic sequence, i.e., it is of the form $S=\langle e,e+d,e+2d,\dots e+t d\rangle$ where $\gcd(e,d)=1$, then $\dmax(S)$ is the number of integer partitions of $e-1$ using the numbers $1,2,\dots,t$. 
\end{prop}

\begin{proof}
First we note that by \cite{BF,MT}, $S$ is additive. Moreover, $B=\langle e,d,2d,\dots,td\rangle = \langle e,d\rangle$ is symmetric with $F(B) = ed-e-d$, see \cite{S}. Thus by Corollary \ref{sym cor}, $\dmax(S)$ is the denumerant of $F(B)+e = ed-e-d +e = (e-1)d$ with respect to the generating set $\mathcal D=\{d,2d,\dots,td\}$. By factoring out $d$, we obtain our result.
\end{proof}

From Proposition \ref{ex3}, we see that if $S$ has embedding dimension 3 with multiplicity $e$, and is generated by an arithmetic sequence, then $\dmax(S)$ is the number of integer partitions of $e-1$ in which all parts are either 1 or 2. This is given by the formula $$\left\lfloor \frac{e-1}{2}+1\right\rfloor = \left\lceil \frac{e}{2}\right\rceil.$$  Thus, for this particular case, we obtain the formula given in Equation 1 in the Introduction using different results than those found in \cite{BHJ}. 

If $S$ has embedding dimension 4 with multiplicity $e$, and is generated by an arithmetic sequence, then $\dmax(S)$ is the number of integer partitions of $e-1$ in which all parts are either 1, 2, or 3. This is given by the formula $$\left[ \frac{(e+2)^2}{12} \right],$$ where $[x]$ denotes the integer nearest to $x$, see \cite{H}. Hence, we obtain a formula of the type given in \cite{BHJ} that is valid for embedding dimension 4, albeit, for only a special class of semigroups.


\begin{thebibliography}{14}

\bibitem{ALF}
Ram\'irez Alfons\'in, J. L., {\em The Diophantine Frobenius Problem}, Oxford Lecture
Series in Mathematics and its Applications {\bf 30}, Oxford University Press,
Oxford, (2005).

\bibitem{ACHP} Amos, J.; Chapman, S. T.; Hine, N.; Paix\~ao, J. Sets of lengths do not characterize numerical monoids. \emph{Integers} {\bf 7} (2007), A50.


\bibitem{AG1} Aguil\'{o}-Gost, F.; Garc\'{i}a-S\'{a}nchez, P. A. Factorization and catenary degree in 3-generated numerical semigroups. European Conference on Combinatorics, Graph Theory and Applications (EuroComb 2009), 157-161, \emph{Electron. Notes Discrete Math.}, {\bf 34}, Elsevier Sci. B. V., Amsterdam.

\bibitem{AG2}  \textemdash, Factoring in embedding dimension three numerical semigroups. \emph{Electron. J. Combin.} {\bf 17} (2010).

\bibitem{Ba}
Barucci, V. (2006). Numerical semigroup algebras. In: Brewer, J., Glaz, S., Heinzer, W., and Olberding, B. Multiplicative Ideal Theory in Commutative Algebra. New York: Springer, pp. 39Ð53.

\bibitem{BCKR}  Bowles, C.; Chapman, S. T.; Kaplan, N.; Reiser, D. On delta sets of numerical monoids. \emph{J. Algebra Appl.} {\bf 5} (2006), {695-718}.


\bibitem{BDF}
V. Barucci, D. Dobbs, M. Fontana, Maximality properties in numerical semigroups and applications to one-dimensional analytically irreducible local domains. Mem. Amer. Math. Soc. 125:vii-77.

\bibitem{BF}
V. Barucci, R. Fr\"oberg, Associated graded rings of one-dimensional analytically irreducible rings, \emph{Journal of Algebra} \textbf{304} (2006) 349--358.

\bibitem{BF2}
\textemdash, Associated graded rings of one-dimensional analytically irreducible rings II, \emph{Journal of Algebra} \textbf{336} (2011) 279--285.

\bibitem{Br} A. Brauer, On a problem of partitions, \emph{Amer. J. Math.} {\bf 64} (1942), 299-312.

\bibitem{Bry}
L. Bryant, Filtered numerical semigroups and applications to one-dimensional rings, PhD thesis, Purdue Univ., 2009.

\bibitem{Bry2}
\textemdash, Goto Numbers of a Numerical Semigroup Ring and the
Gorensteiness of Associated Graded Rings, Comm. Alg. 38 n.6 (2010), 2092 - 2128.

\bibitem{BHJ}
L. Bryant, J. Hamblin, L. Jones, {\emph Maximal denumerant of a numerical semigroup with embedding dimension less than four}, {\it to appear in Journal of Commutative Algebra}.

\bibitem{CDHK} Chapman, S. T.; Daigle, J.; Hoyer, R.; Kaplan, N. Delta sets of numerical monoids using nonminimal sets of generators. \emph{Comm. Algebra} {\bf 38} (2010), {2622-2634}.


\bibitem{CHK} Chapman, S. T.; Hoyer, R.; Kaplan, N. Delta sets of numerical monoids are eventually periodic. \emph{Aequationes Math.} {\bf 77} (2009), {273-279}.


\bibitem{CGLM} Chapman, S. T.; Garc\'{i}a-S\'{a}nchez, P. A.; Llena, D.; Marshall, J. Elements in a numerical semigroup with factorizations of the same length. \emph{Canad. Math. Bull.} {\bf 54} (2011), {39-43}.
\bibitem{DMS}
M. DÕAnna, V. Micale, A. Sammartano, On the Associated Graded Ring of a Semigroup Ring, Journal of Commutative Algebra 3 n.2 (2011), 147- 168.

\bibitem{DMS2}
\textemdash, When the associated graded ring of a semigroup ring is Complete Intersection, in press.

\bibitem{Gar}
A. Garcia, Cohen-Macaulayness of the Associated Graded Ring of a
Semigroup Ring, Comm. Algebra 10 (1982), 393-415.

\bibitem{H} G. Hardy, Some Famous Problems of the Theory of Numbers. Clarendon Press, 1920.

\bibitem{JCR}
J. C. Rosales, On symmetric numerical semigroups, \emph{Journal of Algebra} \textbf{182} (1996), 422--434. MR {\bf 98h}:20099

\bibitem{L}
J. Lipman, Stable ideals and Arf rings, Amer. J. Math. 93 (1971), 649Ð685

\bibitem{MT} S. Molinelli and G. Tamone, On the Hilbert function of certain rings of monomial curves, {\emph J. Pure Appl. Algebra} {\textbf 101} (1995), no. 2, 191Ð206.



\bibitem{R}  Ram\'irez Alfons\'in, J. L. \emph{The Diophantine Frobenius Problem}, Oxford Lecture Series in Mathematics and its Applications {\bf 30}, Oxford University Press, Oxford, (2005).

\bibitem{S} J. Sylvester, Mathematical question with their solutions, \emph{Educational Times} {\textbf 41} (1884), 21.

\end{thebibliography}
\end{document}